\documentclass[11pt]{article}
\usepackage{amsmath, amsthm, amsfonts, amssymb, fullpage, indentfirst, cite}
\theoremstyle{plain}
\newtheorem{proposition}{Proposition}
\newtheorem{lemma}{Lemma}
\newtheorem{theorem}{Theorem}

\theoremstyle{definition}

\newtheorem{assumption}{Assumption}
\newtheorem{algorithm}{Algorithm}
\theoremstyle{remark}

\newenvironment{algorithminit}[1]{\ \\{\em Initialization}: #1\begin{list}{\labelenumi}{\topsep0in\itemsep0in\parsep0in\labelwidth1in\usecounter{enumi}}}{\setcounter{enumii}{\value{enumi}}\end{list}}
\newenvironment{algorithmoper}[1]{{\em Operation}: #1\begin{list}{\labelenumi}{\topsep0in\itemsep0in\parsep0in\labelwidth1in\usecounter{enumi}\setcounter{enumi}{\value{enumii}}}}{\hfill$\blacksquare$\end{list}}

\begin{document}

\title{\Large\bf Gossip Algorithms for Convex Consensus Optimization over Networks\footnote{This work was supported by the National Science Foundation under grant CMMI-0900806.}}
\author{\small\begin{tabular}{cc}Jie Lu, Choon Yik Tang, and Paul R. Regier\footnote{P. R. Regier and T. D. Bow were supported by the National Science Foundation Research Experiences for Undergraduates program under grant EEC-0755011.} & Travis D. Bow\footnotemark[2]\\ School of Electrical and Computer Engineering & Department of Mechanical Engineering\\ University of Oklahoma & Stanford University\\ Norman, OK 73019, USA & Stanford, CA 94305, USA\\ {\sf\{jie.lu-1,cytang,paulregier\}@ou.edu} & {\sf tbow@stanford.edu}\end{tabular}}
\date{\today}
\maketitle

\begin{abstract}
In many applications, nodes in a network desire not only a consensus, but an optimal one. To date, a family of subgradient algorithms have been proposed to solve this problem under general convexity assumptions. This paper shows that, for the scalar case and by assuming a bit more, novel non-gradient-based algorithms with appealing features can be constructed. Specifically, we develop {\em Pairwise Equalizing} (PE) and {\em Pairwise Bisectioning} (PB), two gossip algorithms that solve unconstrained, separable, convex consensus optimization problems over undirected networks with time-varying topologies, where each local function is strictly convex, continuously differentiable, and has a minimizer. We show that PE and PB are easy to implement, bypass limitations of the subgradient algorithms, and produce switched, nonlinear, networked dynamical systems that admit a common Lyapunov function and asymptotically converge. Moreover, PE generalizes the well-known Pairwise Averaging and Randomized Gossip Algorithm, while PB relaxes a requirement of PE, allowing nodes to never share their local functions.
\end{abstract}

\section{Introduction}\label{sec:intro}

Consider an $N$-node multi-hop network, where each node $i$ observes a convex function $f_i$, and all the $N$ nodes wish to determine an optimal consensus $x^*$, which minimizes the sum of the $f_i$'s:
\begin{align}
x^*\in\operatornamewithlimits{arg\,min}_x\sum_{i=1}^Nf_i(x).\label{eq:x*=argminsumf}
\end{align}
Since each node $i$ knows only its own $f_i$, the nodes cannot individually compute the optimal consensus $x^*$ and, thus, must collaborate to do so. This problem of achieving unconstrained, separable, convex consensus optimization has many applications in multi-agent systems and wired/\linebreak[0]wireless/\linebreak[0]social networks, some examples of which can be found in \cite{SonSH05, Rabbat04}.

The current literature offers a large body of work on distributed consensus (see \cite{Olfati-Saber07} for a survey), including a line of research that focuses on solving problem \eqref{eq:x*=argminsumf} for an optimal consensus $x^*$ \cite{Nedic01, Nedic01b, Nedic01c, Rabbat04, Rabbat05, SonSH05, Johansson07, Nedic07, Ram07, Johansson08, Lobel08, Nedic08, Nedic09, Ram09, Ram09b, Ram10}. This line of work has resulted in a family of discrete-time subgradient algorithms, including the {\em incremental} subgradient algorithms \cite{Nedic01, Nedic01b, Nedic01c, Rabbat04, Rabbat05, SonSH05, Johansson07, Ram07, Ram09}, whereby an estimate of $x^*$ is passed around the network, and the {\em non-incremental} ones \cite{Nedic07, Johansson08, Lobel08, Nedic08, Nedic09, Ram09b, Ram10}, whereby each node maintains an estimate of $x^*$ and updates it iteratively by exchanging information with neighbors.

Although the aforementioned subgradient algorithms are capable of solving problem \eqref{eq:x*=argminsumf} under fairly weak assumptions, they suffer from one or more of the following limitations:
\begin{enumerate}
\renewcommand{\theenumi}{L\arabic{enumi}}\itemsep-\parsep
\item {\em Stepsizes}: The algorithms require selection of stepsizes, which may be constant, diminishing, or dynamic. In general, constant stepsizes ensure only convergence to neighborhoods of $x^*$, rather than to $x^*$ itself. Moreover, they present an inevitable trade-off: larger stepsizes tend to yield larger convergence neighborhoods, while smaller ones tend to yield slower convergence. In contrast, diminishing stepsizes typically ensure asymptotic convergence. However, the convergence may be very slow, since the stepsizes may diminish too quickly. Finally, dynamic stepsizes allow shaping of the convergence behavior \cite{Nedic01, Nedic01c}. Unfortunately, their dynamics depend on global information that is often costly to obtain. Hence, selecting appropriate stepsizes is not a trivial task, and inappropriate choices can cause poor performance.
\item {\em Hamiltonian cycle}: Many incremental subgradient algorithms \cite{Nedic01, Nedic01b, Nedic01c, Rabbat04, Rabbat05, SonSH05, Ram07, Ram09} require the nodes to construct and maintain a Hamiltonian cycle (i.e., a closed path that visits every node exactly once) or a pseudo one (i.e., that allows multiple visits), which may be very difficult to carry out, especially in a decentralized, leaderless fashion.
\item {\em Multi-hop transmissions}: Some incremental subgradient algorithms \cite{Nedic01, Nedic01b, Nedic01c} require the node that has the latest estimate of $x^*$ to pass it on to a randomly and equiprobably chosen node in the network. This implies that every node must be aware of all the nodes in the network, and the algorithms must run alongside a routing protocol that enables such passing, which may not always be the case. The fact that the chosen node is typically multiple hops away also implies that these algorithms are communication inefficient, requiring plenty of transmissions (up to the network diameter) just to complete a single iteration.
\item {\em Lack of asymptotic convergence}: A variety of convergence properties have been established for the subgradient algorithms in \cite{Nedic01, Nedic01b, Nedic01c, Rabbat04, Rabbat05, SonSH05, Johansson07, Nedic07, Ram07, Johansson08, Lobel08, Nedic08, Nedic09, Ram09, Ram09b, Ram10}, including error bounds, convergence in expectations, convergence in limit inferiors, convergence rates, etc. In contrast, relatively few asymptotic convergence results have been reported, except for the subgradient algorithms with diminishing or dynamic stepsizes in \cite{Nedic01, Nedic01b, Nedic01c, Ram07, Ram09, Ram09b, Ram10}.
\end{enumerate}

Limitations~L1--L4 facing the subgradient algorithms raise the question of whether it is possible to devise algorithms, which require neither the notion of a stepsize, the construction of a (pseudo-)Hamiltonian cycle, nor the use of a routing protocol for multi-hop transmissions, and yet guarantee asymptotic convergence, bypassing~L1--L4. In this paper, we show that, for the {\em one-dimensional} case and with a few mild assumptions, such algorithms can be constructed. Specifically, instead of letting the network be directed, we assume that it is undirected, with possibly a time-varying topology unknown to any of the nodes. In addition, instead of letting each $f_i$ in \eqref{eq:x*=argminsumf} be convex but not necessarily differentiable, we assume that it is strictly convex, continuously differentiable, and has a minimizer. Based on these assumptions, we develop two gossip-style, distributed asynchronous iterative algorithms, referred to as {\em Pairwise Equalizing} (PE) and {\em Pairwise Bisectioning} (PB), which not only solve problem \eqref{eq:x*=argminsumf} and circumvent limitations~L1--L4, but also are rather easy to implement---although computationally they are more demanding than the subgradient algorithms.

As will be shown in the paper, PE and PB exhibit a number of notable features. First, they produce switched, nonlinear, networked dynamical systems whose state evolves along an invariant manifold whenever nodes gossip with each other. The switched systems are proved, using Lyapunov stability theory, to be asymptotically convergent, as long as the gossiping pattern is sufficiently rich. In particular, we show that the first-order convexity condition can be used to form a common Lyapunov function, as well as to characterize drops in its value after every gossip. Second, PE and PB do not belong to the family of subgradient algorithms as they utilize fundamentally different, non-gradient-based update rules that involve no stepsize. These update rules are synthesized from two simple ideas---{\em conservation} and {\em dissipation}---which are somewhat similar to how Pairwise Averaging \cite{Tsitsiklis84} was conceived back in the 1980s. Indeed, we show that PE reduces to Pairwise Averaging \cite{Tsitsiklis84} and Randomized Gossip Algorithm \cite{Boyd06} when problem \eqref{eq:x*=argminsumf} specializes to an averaging problem. Finally, PE requires one-time sharing of the $f_i$'s between gossiping nodes, which may be costly or impermissible in some applications. This requirement is eliminated by PB at the expense of more communications per iteration.

\section{Problem Formulation}\label{sec:probform}

Consider a multi-hop network consisting of $N\ge2$ nodes, connected by bidirectional links in a time-varying topology. The network is modeled as an undirected graph $\mathcal{G}(k)=(\mathcal{V},\mathcal{E}(k))$, where $k\in\mathbb{N}=\{0,1,2,\ldots\}$ denotes time, $\mathcal{V}=\{1,2,\ldots,N\}$ represents the set of $N$ nodes, and $\mathcal{E}(k)\subset\{\{i,j\}:i,j\in\mathcal{V},i\neq j\}$ represents the nonempty set of links at time $k$. Any two nodes $i,j\in\mathcal{V}$ are one-hop neighbors and can communicate at time $k\in\mathbb{N}$ if and only if $\{i,j\}\in\mathcal{E}(k)$.

Suppose, at time $k=0$, each node $i\in\mathcal{V}$ observes a function $f_i:\mathcal{X}\rightarrow\mathbb{R}$, which maps a nonempty open interval $\mathcal{X}\subset\mathbb{R}$ to $\mathbb{R}$, and which satisfies the following assumption:

\begin{assumption}\label{asm:fi}
For each $i\in\mathcal{V}$, the function $f_i$ is strictly convex, continuously differentiable, and has a minimizer $x_i^*\in\mathcal{X}$.
\end{assumption}

Suppose, upon observing the $f_i$'s, all the $N$ nodes wish to solve the following unconstrained, separable, convex optimization problem:
\begin{align}
\min_{x\in\mathcal{X}}F(x),\label{eq:minF}
\end{align}
where the function $F:\mathcal{X}\rightarrow\mathbb{R}$ is defined as $F(x)=\sum_{i\in\mathcal{V}}f_i(x)$. Clearly, $F$ is strictly convex and continuously differentiable. To show that $F$ has a unique minimizer in $\mathcal{X}$ so that problem \eqref{eq:minF} is well-posed, let $f_i':\mathcal{X}\rightarrow\mathbb{R}$ and $F':\mathcal{X}\rightarrow\mathbb{R}$ denote the derivatives of $f_i$ and $F$, respectively, and consider the following lemma and proposition:

\begin{lemma}\label{lem:exisuniqz}
Let $g_i:\mathcal{X}\rightarrow\mathbb{R}$ be a strictly increasing and continuous function and $z_i\in\mathcal{X}$ for $i=1,2,\ldots,n$. Then, there exists a unique $z\in\mathcal{X}$ such that $\sum_{i=1}^ng_i(z)=\sum_{i=1}^ng_i(z_i)$. Moreover, $z\in[\min_{i\in\{1,2,\ldots,n\}}z_i,\max_{i\in\{1,2,\ldots,n\}}z_i]$.
\end{lemma}

\begin{proof}
Since $g_i$ is strictly increasing and continuous $\forall i\in\{1,2,\ldots,n\}$, so is $\sum_{i=1}^ng_i:\mathcal{X}\rightarrow\mathbb{R}$. Thus, $\sum_{i=1}^ng_i(\min_{j\in\{1,2,\ldots,n\}}z_j)\le\sum_{i=1}^ng_i(z_i)\le\sum_{i=1}^ng_i(\max_{j\in\{1,2,\ldots,n\}}z_j)$. It follows from the Intermediate Value Theorem that there exists a unique $z\in\mathcal{X}$ such that $\sum_{i=1}^ng_i(z)=\sum_{i=1}^ng_i(z_i)$, and that $z\in[\min_{i\in\{1,2,\ldots,n\}}z_i,\max_{i\in\{1,2,\ldots,n\}}z_i]$.
\end{proof}

\begin{proposition}\label{pro:exisuniqx*}
With Assumption~\ref{asm:fi}, there exists a unique $x^*\in\mathcal{X}$, which satisfies $F'(x^*)=0$, minimizes $F$ over $\mathcal{X}$, and solves problem \eqref{eq:minF}, i.e., $x^*=\operatornamewithlimits{arg\,min}_{x\in\mathcal{X}}F(x)$.
\end{proposition}

\begin{proof}
By Assumption~\ref{asm:fi}, for every $i\in\mathcal{V}$, $f_i'$ is strictly increasing and continuous. By Lemma~\ref{lem:exisuniqz}, there exists a unique $x^*\in\mathcal{X}$ such that $\sum_{i\in\mathcal{V}}f_i'(x^*)=\sum_{i\in\mathcal{V}}f_i'(x_i^*)$. Since $F'=\sum_{i\in\mathcal{V}}f_i'$ and $f_i'(x_i^*)=0$ $\forall i\in\mathcal{V}$, $F'(x^*)=0$. Since $F$ is strictly convex, $x^*$ minimizes $F$ over $\mathcal{X}$, solving \eqref{eq:minF}.
\end{proof}

Given the above, the goal is to construct a distributed asynchronous iterative algorithm free of limitations~L1--L4, with which each node can asymptotically determine the unknown optimizer $x^*$.

\section{Pairwise Equalizing}\label{sec:PE}

In this section, we develop a gossip algorithm having the aforementioned features.

Suppose, at time $k=0$, each node $i\in\mathcal{V}$ creates a state variable $\hat{x}_i\in\mathcal{X}$ in its local memory, which represents its estimate of $x^*$. Also suppose, at each subsequent time $k\in\mathbb{P}=\{1,2,\ldots\}$, an iteration, called {\em iteration $k$}, takes place. Let $\hat{x}_i(0)$ represent the initial value of $\hat{x}_i$, and $\hat{x}_i(k)$ its value upon completing each iteration $k\in\mathbb{P}$. With this setup, the goal may be stated as
\begin{align}
\lim_{k\rightarrow\infty}\hat{x}_i(k)=x^*,\quad\forall i\in\mathcal{V}.\label{eq:limxh=x*}
\end{align}

To design an algorithm that guarantees \eqref{eq:limxh=x*}, consider a {\em conservation condition}
\begin{align}
\sum_{i\in\mathcal{V}}f_i'(\hat{x}_i(k))=0,\quad\forall k\in\mathbb{N},\label{eq:sumf'xh=0}
\end{align}
which says that the $\hat{x}_i(k)$'s evolve in a way that the sum of the derivatives $f_i'$'s, evaluated at the $\hat{x}_i(k)$'s, is always conserved at zero. Moreover, consider a {\em dissipation condition}
\begin{align}
\lim_{k\rightarrow\infty}\hat{x}_i(k)=\tilde{x},\quad\forall i\in\mathcal{V},\;\text{for some}\;\tilde{x}\in\mathcal{X},\label{eq:limxh=xt}
\end{align}
which says that the $\hat{x}_i(k)$'s gradually dissipate their differences and asymptotically achieve some arbitrary consensus $\tilde{x}\in\mathcal{X}$. Note that if \eqref{eq:sumf'xh=0} is met, then $\lim_{k\rightarrow\infty}\sum_{i\in\mathcal{V}}f_i'(\hat{x}_i(k))=\lim_{k\rightarrow\infty}0=0$. If, in addition, \eqref{eq:limxh=xt} is met, then due to the continuity of every $f_i'$, $\sum_{i\in\mathcal{V}}\lim_{k\rightarrow\infty}f_i'(\hat{x}_i(k))=\sum_{i\in\mathcal{V}}f_i'(\lim_{k\rightarrow\infty}\hat{x}_i(k))=\sum_{i\in\mathcal{V}}f_i'(\tilde{x})=F'(\tilde{x})$. Because $\lim_{k\rightarrow\infty}f_i'(\hat{x}_i(k))$ exists for every $i\in\mathcal{V}$, $\lim_{k\rightarrow\infty}\sum_{i\in\mathcal{V}}f_i'(\hat{x}_i(k))=\sum_{i\in\mathcal{V}}\lim_{k\rightarrow\infty}f_i'(\hat{x}_i(k))$. Combining the above, we obtain $F'(\tilde{x})=0$. From Proposition~\ref{pro:exisuniqx*}, we see that the arbitrary consensus $\tilde{x}$ must be the unknown optimizer $x^*$, i.e., $\tilde{x}=x^*$, so that \eqref{eq:limxh=x*} holds. Therefore, to design an algorithm that ensures \eqref{eq:limxh=x*}---where $x^*$ explicitly appears, it suffices to make the algorithm satisfy both the conservation and dissipation conditions \eqref{eq:sumf'xh=0} and \eqref{eq:limxh=xt}---where $x^*$ is implicitly encoded.

To this end, observe that \eqref{eq:sumf'xh=0} holds if and only if the $\hat{x}_i(0)$'s are such that $\sum_{i\in\mathcal{V}}f_i'(\hat{x}_i(0))=0$, and the $\hat{x}_i(k)$'s are related to the $\hat{x}_i(k-1)$'s through
\begin{align}
\sum_{i\in\mathcal{V}}f_i'(\hat{x}_i(k))=\sum_{i\in\mathcal{V}}f_i'(\hat{x}_i(k-1)),\quad\forall k\in\mathbb{P}.\label{eq:sumf'xh=sumf'xh}
\end{align}
To satisfy $\sum_{i\in\mathcal{V}}f_i'(\hat{x}_i(0))=0$, it suffices that each node $i\in\mathcal{V}$ computes $x_i^*$ on its own and sets
\begin{align}
\hat{x}_i(0)=x_i^*,\quad\forall i\in\mathcal{V},\label{eq:xh0=x*}
\end{align}
since $f_i'(x_i^*)=0$. To satisfy \eqref{eq:sumf'xh=sumf'xh}, consider a gossip algorithm, whereby at each iteration $k\in\mathbb{P}$, a pair $u(k)=\{u_1(k),u_2(k)\}\in\mathcal{E}(k)$ of one-hop neighbors $u_1(k)$ and $u_2(k)$ gossip and update their $\hat{x}_{u_1(k)}(k)$ and $\hat{x}_{u_2(k)}(k)$, while the rest of the $N$ nodes stay idle, i.e.,
\begin{align}
\hat{x}_i(k)=\hat{x}_i(k-1),\quad\forall k\in\mathbb{P},\;\forall i\in\mathcal{V}-u(k).\label{eq:xh=xh/u}
\end{align}
With \eqref{eq:xh=xh/u}, equation \eqref{eq:sumf'xh=sumf'xh} simplifies to
\begin{align}
f_{u_1(k)}'(\hat{x}_{u_1(k)}(k))+f_{u_2(k)}'(\hat{x}_{u_2(k)}(k))=f_{u_1(k)}'(\hat{x}_{u_1(k)}(k-1))+f_{u_2(k)}'(\hat{x}_{u_2(k)}(k-1)),\quad\forall k\in\mathbb{P}.\label{eq:f'xhf'xh=f'xhf'xh}
\end{align}
Hence, all that is needed for \eqref{eq:sumf'xh=sumf'xh} to hold is a gossip between nodes $u_1(k)$ and $u_2(k)$ to share their $f_{u_1(k)}$, $f_{u_2(k)}$, $\hat{x}_{u_1(k)}(k-1)$, and/or $\hat{x}_{u_2(k)}(k-1)$, followed by a joint update of their $\hat{x}_{u_1(k)}(k)$ and $\hat{x}_{u_2(k)}(k)$, which ensures \eqref{eq:f'xhf'xh=f'xhf'xh}.

Obviously, \eqref{eq:f'xhf'xh=f'xhf'xh} alone does not uniquely determine $\hat{x}_{u_1(k)}(k)$ and $\hat{x}_{u_2(k)}(k)$. This suggests that the available degree of freedom may be used to account for the dissipation condition \eqref{eq:limxh=xt}. Unlike the conservation condition \eqref{eq:sumf'xh=0}, however, \eqref{eq:limxh=xt} is about where the $\hat{x}_i(k)$'s should approach as $k\rightarrow\infty$, which nodes $u_1(k)$ and $u_2(k)$ cannot guarantee themselves since they are only responsible for two of the $N$ $\hat{x}_i(k)$'s. Nevertheless, given that all the $N$ $\hat{x}_i(k)$'s should approach the {\em same} limit, nodes $u_1(k)$ and $u_2(k)$ can help make this happen by imposing an {\em equalizing condition}
\begin{align}
\hat{x}_{u_1(k)}(k)=\hat{x}_{u_2(k)}(k),\quad\forall k\in\mathbb{P}.\label{eq:xh=xh}
\end{align}
With \eqref{eq:xh=xh} added, there are now two equations with two variables, providing nodes $u_1(k)$ and $u_2(k)$ a chance to uniquely determine $\hat{x}_{u_1(k)}(k)$ and $\hat{x}_{u_2(k)}(k)$ from \eqref{eq:f'xhf'xh=f'xhf'xh} and \eqref{eq:xh=xh}.

The following proposition asserts that \eqref{eq:f'xhf'xh=f'xhf'xh} and \eqref{eq:xh=xh} always have a unique solution, so that the evolution of the $\hat{x}_i(k)$'s is well-defined:

\begin{proposition}\label{pro:xhwelldef}
With Assumption~\ref{asm:fi} and \eqref{eq:xh0=x*}--\eqref{eq:xh=xh}, $\hat{x}_i(k)$ $\forall k\in\mathbb{N}$ $\forall i\in\mathcal{V}$ are well-defined, i.e., unambiguous and in $\mathcal{X}$. Moreover, $[\min\limits_{i\in\mathcal{V}}\hat{x}_i(k),\max\limits_{i\in\mathcal{V}}\hat{x}_i(k)]\subset[\min\limits_{i\in\mathcal{V}}\hat{x}_i(k-1),\max\limits_{i\in\mathcal{V}}\hat{x}_i(k-1)]$ $\forall k\in\mathbb{P}$.
\end{proposition}

\begin{proof}
By induction on $k\in\mathbb{N}$. By Assumption~\ref{asm:fi} and \eqref{eq:xh0=x*}, $\hat{x}_i(0)$ $\forall i\in\mathcal{V}$ are unambiguous and in $\mathcal{X}$. Next, let $k\in\mathbb{P}$ and suppose $\hat{x}_i(k-1)$ $\forall i\in\mathcal{V}$ are unambiguous and in $\mathcal{X}$. We show that so are $\hat{x}_i(k)$ $\forall i\in\mathcal{V}$. From \eqref{eq:xh=xh/u}, $\hat{x}_i(k)$ $\forall i\in\mathcal{V}-u(k)$ are unambiguous and in $\mathcal{X}$. To show that so are $\hat{x}_{u_1(k)}(k)$ and $\hat{x}_{u_2(k)}(k)$, we show that \eqref{eq:f'xhf'xh=f'xhf'xh} and \eqref{eq:xh=xh} have a unique solution $(\hat{x}_{u_1(k)}(k),\hat{x}_{u_2(k)}(k))\in\mathcal{X}^2$. By Lemma~\ref{lem:exisuniqz}, there is a unique $z\in\mathcal{X}$ such that
\begin{align}
f_{u_1(k)}'(z)+f_{u_2(k)}'(z)=f_{u_1(k)}'(\hat{x}_{u_1(k)}(k-1))+f_{u_2(k)}'(\hat{x}_{u_2(k)}(k-1)),\label{eq:f'zf'z=f'xhf'xh}
\end{align}
which satisfies $z\in[\min_{i\in u(k)}\hat{x}_i(k-1),\max_{i\in u(k)}\hat{x}_i(k-1)]$. Setting $\hat{x}_{u_1(k)}(k)=\hat{x}_{u_2(k)}(k)=z$, we see that $(\hat{x}_{u_1(k)}(k),\hat{x}_{u_2(k)}(k))$ is a solution to \eqref{eq:f'xhf'xh=f'xhf'xh} and \eqref{eq:xh=xh}, confirming the existence. Now let $(a_1,a_2)\in\mathcal{X}^2$ and $(b_1,b_2)\in\mathcal{X}^2$ be two solutions of \eqref{eq:f'xhf'xh=f'xhf'xh} and \eqref{eq:xh=xh}. Then, due to \eqref{eq:xh=xh}, \eqref{eq:f'xhf'xh=f'xhf'xh}, and Lemma~\ref{lem:exisuniqz}, we have $a_1=a_2=b_1=b_2$, confirming the uniqueness. Therefore, $\hat{x}_i(k)$ $\forall i\in\mathcal{V}$ are well-defined as desired. Finally, the second statement follows from \eqref{eq:xh=xh/u} and the fact that $\hat{x}_{u_1(k)}(k)=\hat{x}_{u_2(k)}(k)\in[\min_{i\in u(k)}\hat{x}_i(k-1),\max_{i\in u(k)}\hat{x}_i(k-1)]$ $\forall k\in\mathbb{P}$.
\end{proof}

Proposition~\ref{pro:xhwelldef} calls for a few remarks. First, the interval $[\min_{i\in\mathcal{V}}\hat{x}_i(k),\max_{i\in\mathcal{V}}\hat{x}_i(k)]$ can only shrink or remain unchanged over time $k$. While this does not guarantee the dissipation condition \eqref{eq:limxh=xt}, it shows that the $\hat{x}_i(k)$'s are ``trying'' to converge and are, at the very least, bounded even if $\mathcal{X}$ is not. Second, the proofs of Proposition~\ref{pro:xhwelldef} and Lemma~\ref{lem:exisuniqz} suggest a simple, practical procedure for nodes $u_1(k)$ and $u_2(k)$ to solve \eqref{eq:f'xhf'xh=f'xhf'xh} and \eqref{eq:xh=xh} for $(\hat{x}_{u_1(k)}(k),\hat{x}_{u_2(k)}(k))$: apply a numerical {\em root-finding method}, such as the {\em bisection method} with initial bracket $[\min_{i\in u(k)}\hat{x}_i(k-1),\max_{i\in u(k)}\hat{x}_i(k-1)]$, to solve \eqref{eq:f'zf'z=f'xhf'xh} for the unique $z$ and then set $\hat{x}_{u_1(k)}(k)=\hat{x}_{u_2(k)}(k)=z$. Finally, since \eqref{eq:f'zf'z=f'xhf'xh} always has a unique solution $z$, we can eliminate $z$ and write
\begin{align}
\hat{x}_{u_1(k)}(k)=\hat{x}_{u_2(k)}(k)=(f_{u_1(k)}'+f_{u_2(k)}')^{-1}(f_{u_1(k)}'(\hat{x}_{u_1(k)}(k-1))+f_{u_2(k)}'(\hat{x}_{u_2(k)}(k-1))),\quad\forall k\in\mathbb{P},\label{eq:xh=xh=f'f'invf'xhf'xh}
\end{align}
where $(f_i'+f_j')^{-1}:(f_i'+f_j')(\mathcal{X})\rightarrow\mathcal{X}$ denotes the inverse of the injective function $f_i'+f_j'$ with its codomain restricted to its range.

Expressions \eqref{eq:xh0=x*}, \eqref{eq:xh=xh/u}, and \eqref{eq:xh=xh=f'f'invf'xhf'xh} collectively define a gossip-style, distributed asynchronous iterative algorithm that yields a switched, nonlinear, networked dynamical system
\begin{align}
\hat{x}_i(k)=\begin{cases}(\sum_{j\in u(k)}f_j')^{-1}(\sum_{j\in u(k)}f_j'(\hat{x}_j(k-1))), & \text{if $i\in u(k)$},\\ \hat{x}_i(k-1), & \text{otherwise},\end{cases}\quad\forall k\in\mathbb{P},\;\forall i\in\mathcal{V},\label{eq:xh=sumf'invsumf'xhifiinuxh}
\end{align}
with initial condition \eqref{eq:xh0=x*}, and with $(u(k))_{k=1}^\infty$ representing the sequence of gossiping nodes that trigger the switchings. As this algorithm ensures the conservation condition \eqref{eq:sumf'xh=0}, the state trajectory $(\hat{x}_1(k),\hat{x}_2(k),\ldots,\hat{x}_N(k))$ must remain on an $(N-1)$-dimensional manifold $\mathcal{M}=\{(x_1,x_2,\ldots,x_N)\in\mathcal{X}^N:\sum_{i\in\mathcal{V}}f_i'(x_i)=0\}\subset\mathcal{X}^N\subset\mathbb{R}^N$, making $\mathcal{M}$ an invariant set. Given that the algorithm involves repeated, pairwise equalizing of the $\hat{x}_i(k)$'s, we refer to it as {\em Pairwise Equalizing} (PE). PE may be expressed in a compact algorithmic form as follows:

\begin{algorithm}[Pairwise Equalizing]\label{alg:PE}
\begin{algorithminit}{}
\item Each node $i\in\mathcal{V}$ computes $x_i^*\in\mathcal{X}$, creates a variable $\hat{x}_i\in\mathcal{X}$, and sets $\hat{x}_i\leftarrow x_i^*$.
\end{algorithminit}
\begin{algorithmoper}{At each iteration:}
\item A node with one or more one-hop neighbors, say, node $i$, initiates the iteration and selects a one-hop neighbor, say, node $j$, to gossip. Nodes $i$ and $j$ select one of two ways to gossip by labeling themselves as either nodes $a$ and $b$, or nodes $b$ and $a$, respectively, where $\{a,b\}=\{i,j\}$. If node $b$ does not know $f_a$, node $a$ transmits $f_a$ to node $b$. Node $a$ transmits $\hat{x}_a$ to node $b$. Node $b$ sets $\hat{x}_b\leftarrow(f_a'+f_b')^{-1}(f_a'(\hat{x}_a)+f_b'(\hat{x}_b))$ and transmits $\hat{x}_b$ to node $a$. Node $a$ sets $\hat{x}_a\leftarrow\hat{x}_b$.
\end{algorithmoper}
\end{algorithm}

Due to space limitations, we omit remarks concerning the execution of Algorithm~\ref{alg:PE} and refer the reader to an earlier, conference version of this paper \cite{LuJ10}.

Notice that PE does not rely on a stepsize parameter to execute, nor does it require the construction of a (pseudo-)Hamiltonian cycle, as well as the concurrent use of a routing protocol for multi-hop transmissions. Indeed, all it essentially needs is that every node is capable of applying a root-finding method, maintaining a list of its one-hop neighbors, and remembering the functions it learns along the way. Therefore, PE overcomes limitations~L1--L3, while being rather easy to implement---although computationally it is more demanding than the subgradient algorithms.

To show that PE asymptotically converges and, thus, circumvents~L4, let $\mathbf{x}^*=(x^*,x^*,\ldots,x^*)$ and $\mathbf{x}(k)=(\hat{x}_1(k),\hat{x}_2(k),\ldots,\hat{x}_N(k))$. Then, from Propositions~\ref{pro:exisuniqx*} and~\ref{pro:xhwelldef}, $\mathbf{x}^*\in\mathcal{X}^N$ and $\mathbf{x}(k)\in\mathcal{X}^N$ $\forall k\in\mathbb{N}$. In addition, due to \eqref{eq:xh=sumf'invsumf'xhifiinuxh}, if $\mathbf{x}(k)=\mathbf{x}^*$ for some $k\in\mathbb{N}$, then $\mathbf{x}(\ell)=\mathbf{x}^*$ $\forall\ell>k$. Hence, $\mathbf{x}^*$ is an equilibrium point of the system \eqref{eq:xh=sumf'invsumf'xhifiinuxh}. To show that $\lim_{k\rightarrow\infty}\mathbf{x}(k)=\mathbf{x}^*$, i.e., \eqref{eq:limxh=x*} holds, we seek to construct a Lyapunov function. To this end, recall that for any strictly convex and differentiable function $f:\mathcal{X}\rightarrow\mathbb{R}$, the first-order convexity condition says that
\begin{align}
f(y)\ge f(x)+f'(x)(y-x),\quad\forall x,y\in\mathcal{X},\label{eq:fy>=fxf'xyx}
\end{align}
where the equality holds if and only if $x=y$. This suggests the following Lyapunov function candidate $V:\mathcal{X}^N\subset\mathbb{R}^N\rightarrow\mathbb{R}$, which exploits the convexity of the $f_i$'s:
\begin{align}
V(\mathbf{x}(k))=\sum_{i\in\mathcal{V}}f_i(x^*)-f_i(\hat{x}_i(k))-f_i'(\hat{x}_i(k))(x^*-\hat{x}_i(k)).\label{eq:V=sumfx*fxhf'xhx*xh}
\end{align}
Notice that $V$ in \eqref{eq:V=sumfx*fxhf'xhx*xh} is well-defined. Moreover, due to Assumption~\ref{asm:fi} and \eqref{eq:fy>=fxf'xyx}, $V$ is continuous and positive definite with respect to $\mathbf{x}^*$, i.e., $V(\mathbf{x}(k))\ge0$ $\forall\mathbf{x}(k)\in\mathcal{X}^N$, where the equality holds if and only if $\mathbf{x}(k)=\mathbf{x}^*$. Therefore, to prove \eqref{eq:limxh=x*}, it suffices to show that
\begin{align}
\lim_{k\rightarrow\infty}V(\mathbf{x}(k))=0.\label{eq:limV=0}
\end{align}

The following lemma represents the first step toward establishing \eqref{eq:limV=0}:

\begin{lemma}\label{lem:PEVnonincr}
Consider the use of PE described in Algorithm~\ref{alg:PE}. Suppose Assumption~\ref{asm:fi} holds. Then, for any given $(u(k))_{k=1}^\infty$, $(V(\mathbf{x}(k)))_{k=0}^\infty$ is non-increasing and satisfies
\begin{align}
V(\mathbf{x}(k))-V(\mathbf{x}(k-1))=-\sum_{i\in u(k)}f_i(\hat{x}_i(k))-f_i(\hat{x}_i(k-1))-f_i'(\hat{x}_i(k-1))(\hat{x}_i(k)&-\hat{x}_i(k-1)),\nonumber\\
&\quad\forall k\in\mathbb{P}.\label{eq:VV=sumfxhfxhf'xhxhxh}
\end{align}
\end{lemma}

\begin{proof}
Let $(u(k))_{k=1}^\infty$ be given. Then, from \eqref{eq:V=sumfx*fxhf'xhx*xh} and \eqref{eq:xh=sumf'invsumf'xhifiinuxh}, we have $V(\mathbf{x}(k))-V(\mathbf{x}(k-1))=-\sum_{i\in u(k)}f_i(\hat{x}_i(k))-f_i(\hat{x}_i(k-1))+f_i'(\hat{x}_i(k))x^*-f_i'(\hat{x}_i(k-1))x^*-f_i'(\hat{x}_i(k))\hat{x}_i(k)+f_i'(\hat{x}_i(k-1))\hat{x}_i(k-1)$ $\forall k\in\mathbb{P}$. Due to \eqref{eq:xh=sumf'invsumf'xhifiinuxh}, $-\sum_{i\in u(k)}f_i'(\hat{x}_i(k))x^*$ cancels $\sum_{i\in u(k)}f_i'(\hat{x}_i(k-1))x^*$, while $\sum_{i\in u(k)}f_i'(\hat{x}_i(k))\hat{x}_i(k)$ becomes $\sum_{i\in u(k)}f_i'(\hat{x}_i(k-1))\hat{x}_i(k)$. This proves \eqref{eq:VV=sumfxhfxhf'xhxhxh}. Note that the right-hand side of \eqref{eq:VV=sumfxhfxhf'xhxhxh} is nonpositive due to \eqref{eq:fy>=fxf'xyx}. Hence, $(V(\mathbf{x}(k)))_{k=0}^\infty$ is non-increasing.
\end{proof}

Lemma~\ref{lem:PEVnonincr} has several implications. First, upon completing each iteration $k\in\mathbb{P}$ by {\em any} two nodes $u_1(k)$ and $u_2(k)$, the value of $V$ must either decrease or, at worst, stay the same, where the latter occurs if and only if $\hat{x}_{u_1(k)}(k-1)=\hat{x}_{u_2(k)}(k-1)$. Second, since $(V(\mathbf{x}(k)))_{k=0}^\infty$ is non-increasing irrespective of $(u(k))_{k=1}^\infty$, $V$ in \eqref{eq:V=sumfx*fxhf'xhx*xh} may be regarded as a {\em common} Lyapunov function for the nonlinear switched system \eqref{eq:xh=sumf'invsumf'xhifiinuxh}, which has as many as $\frac{N(N-1)}{2}$ different dynamics, corresponding to the $\frac{N(N-1)}{2}$ possible gossiping pairs. Finally, the first-order convexity condition \eqref{eq:fy>=fxf'xyx} can be used not only to form the common Lyapunov function $V$, but also to characterize drops in its value in \eqref{eq:VV=sumfxhfxhf'xhxhxh} after every gossip. This is akin to how quadratic functions may be used to form a common Lyapunov function $V(k)=x^T(k)Px(k)$ for a linear switched system $x(k+1)=A(k)x(k)$, $A(k)\in\{A_1,A_2,\ldots,A_M\}$, as well as to characterize drops in $V(k)$ via $V(k+1)-V(k)=x^T(k)(A_i^TPA_i-P)x(k)=-x^T(k)Q_ix(k)$. Indeed, as we will show later, when problem \eqref{eq:minF} specializes to an averaging problem, where the nonlinear switched system \eqref{eq:xh=sumf'invsumf'xhifiinuxh} becomes linear, both $V$ and its drop become quadratic functions.

As $(V(\mathbf{x}(k)))_{k=0}^\infty$ is nonnegative and non-increasing, $\lim_{k\rightarrow\infty}V(\mathbf{x}(k))$ exists and is nonnegative. This, however, is insufficient for us to conclude that $\lim_{k\rightarrow\infty}V(\mathbf{x}(k))=0$, since, for some pathological gossiping patterns, $\lim_{k\rightarrow\infty}V(\mathbf{x}(k))$ can be positive (see \cite{LuJ10} for examples). Thus, some restrictions must be imposed on the gossiping pattern, in order to establish \eqref{eq:limV=0}. To this end, let $\mathcal{E}_\infty=\{\{i,j\}:u(k)=\{i,j\}\;\text{for infinitely many}\;k\in\mathbb{P}\}$, so that a link $\{i,j\}$ is in $\mathcal{E}_\infty$ if and only if nodes $i$ and $j$ gossip with each other infinitely often. Then, we may state the following restriction on the gossiping pattern, which was first adopted in \cite{Tsitsiklis84} and is not difficult to satisfy in practice \cite{LuJ10}:

\begin{assumption}\label{asm:PEnetconn}
The sequence $(u(k))_{k=1}^\infty$ is such that the graph $(\mathcal{V},\mathcal{E}_\infty)$ is connected.
\end{assumption}

The following theorem says that, under Assumption~\ref{asm:PEnetconn} on the gossiping pattern, PE ensures asymptotic convergence of all the $\hat{x}_i(k)$'s to $x^*$, circumventing limitation~L4:

\begin{theorem}\label{thm:PEasymconv}
Consider the use of PE described in Algorithm~\ref{alg:PE}. Suppose Assumptions~\ref{asm:fi} and~\ref{asm:PEnetconn} hold. Then, \eqref{eq:limV=0} and \eqref{eq:limxh=x*} hold.
\end{theorem}

\begin{proof}
See Appendix~\ref{ssec:proofthmPEasymconv}.
\end{proof}

Finally, we point out that the above results may be viewed as a natural generalization of some known results in distributed averaging. Consider a special case where each node $i\in\mathcal{V}$ observes not an arbitrary function $f_i$, but a quadratic one of the form $f_i(x)=\frac{1}{2}(x-y_i)^2+c_i$ with domain $\mathcal{X}=\mathbb{R}$ and parameters $y_i,c_i\in\mathbb{R}$. In this case, finding the unknown optimizer $x^*$ amounts to calculating the network-wide average $\frac{1}{N}\sum_{i\in\mathcal{V}}y_i$ of the node ``observations'' $y_i$'s, so that the convex optimization problem \eqref{eq:minF} becomes an averaging problem. In addition, initializing the node estimates $\hat{x}_i(0)$'s simply means setting them to the $y_i$'s, and equalizing $\hat{x}_{u_1(k)}(k)$ and $\hat{x}_{u_2(k)}(k)$ simply means averaging them, so that PE reduces to Pairwise Averaging \cite{Tsitsiklis84} and Randomized Gossip Algorithm \cite{Boyd06}. Moreover, the invariant manifold $\mathcal{M}$ becomes the invariant hyperplane $\mathcal{M}=\{(x_1,x_2,\ldots,x_N)\in\mathbb{R}^N:\sum_{i\in\mathcal{V}}x_i=\sum_{i\in\mathcal{V}}y_i\}$ in distributed averaging. Furthermore, both the common Lyapunov function $V$ in \eqref{eq:V=sumfx*fxhf'xhx*xh} and its drop in \eqref{eq:VV=sumfxhfxhf'xhxhxh} take a quadratic form: $V(\mathbf{x}(k))=\frac{1}{2}(\mathbf{x}(k)-\mathbf{x}^*)^T(\mathbf{x}(k)-\mathbf{x}^*)$ and $V(\mathbf{x}(k))-V(\mathbf{x}(k-1))=-\frac{1}{2}\mathbf{x}^T(k-1)Q_{u(k)}\mathbf{x}(k-1)$ $\forall k\in\mathbb{P}$, where $Q_{\{i,j\}}\in\mathbb{R}^{N\times N}$ is a symmetric positive semidefinite matrix whose $ii$ and $jj$ entries are $\frac{1}{2}$, $ij$ and $ji$ entries are $-\frac{1}{2}$, and all other entries are zero. Therefore, the first-order-convexity-condition-based Lyapunov function \eqref{eq:V=sumfx*fxhf'xhx*xh} generalizes the quadratic Lyapunov function in distributed averaging.

\section{Pairwise Bisectioning}\label{sec:PB}

Although PE solves problem \eqref{eq:minF} and bypasses~L1--L4, it requires one-time, one-way sharing of the $f_i$'s between gossiping nodes, which may be costly for certain $f_i$'s, or impermissible for security and privacy reasons. In this section, we develop another gossip algorithm that eliminates this requirement at the expense of more real-number transmissions per iteration.

Note that PE can be traced back to four defining equations \eqref{eq:xh0=x*}--\eqref{eq:xh=xh}, and that its drawback of having to share the $f_i$'s stems from having to solve \eqref{eq:f'xhf'xh=f'xhf'xh} and \eqref{eq:xh=xh}. To overcome this drawback, consider a gossip algorithm satisfying \eqref{eq:xh0=x*}--\eqref{eq:f'xhf'xh=f'xhf'xh} and a new condition but not \eqref{eq:xh=xh}. Assuming, without loss of generality, that $\hat{x}_{u_1(k)}(k-1)\le\hat{x}_{u_2(k)}(k-1)$ $\forall k\in\mathbb{P}$, this new condition can be stated as
\begin{align}
\hat{x}_{u_1(k)}(k-1)\le\hat{x}_{u_1(k)}(k)\le\hat{x}_{u_2(k)}(k)\le\hat{x}_{u_2(k)}(k-1),\quad\forall k\in\mathbb{P}.\label{eq:xh<=xh<=xh<=xh}
\end{align}
Termed as the {\em approaching condition}, \eqref{eq:xh<=xh<=xh<=xh} says that at each iteration $k\in\mathbb{P}$, nodes $u_1(k)$ and $u_2(k)$ force $\hat{x}_{u_1(k)}(k)$ and $\hat{x}_{u_2(k)}(k)$ to approach each other while preserving their order. Observe that the approaching condition \eqref{eq:xh<=xh<=xh<=xh} includes the equalizing condition \eqref{eq:xh=xh} as a special case. Furthermore, unlike \eqref{eq:f'xhf'xh=f'xhf'xh} and \eqref{eq:xh=xh}, \eqref{eq:f'xhf'xh=f'xhf'xh} and \eqref{eq:xh<=xh<=xh<=xh} do not uniquely determine $\hat{x}_{u_1(k)}(k)$ and $\hat{x}_{u_2(k)}(k)$. Rather, they allow $\hat{x}_{u_1(k)}(k)$ and $\hat{x}_{u_2(k)}(k)$ to increase gradually from $\hat{x}_{u_1(k)}(k-1)$ and decrease accordingly from $\hat{x}_{u_2(k)}(k-1)$, respectively, until the two become equal.

The following lemma characterizes the impact of the non-uniqueness on the value of $V$:

\begin{lemma}\label{lem:PBVnonincr}
Consider \eqref{eq:xh0=x*}--\eqref{eq:f'xhf'xh=f'xhf'xh} and \eqref{eq:xh<=xh<=xh<=xh}. Suppose Assumption~\ref{asm:fi} holds. Then, for any given $(u(k))_{k=1}^\infty$, $(V(\mathbf{x}(k)))_{k=0}^\infty$ is non-increasing. Moreover, for any given $k\in\mathbb{P}$ and $\mathbf{x}(k-1)\in\mathcal{X}^N$, $V(\mathbf{x}(k))$ strictly increases with $\hat{x}_{u_2(k)}(k)-\hat{x}_{u_1(k)}(k)$ over $[0,\hat{x}_{u_2(k)}(k-1)-\hat{x}_{u_1(k)}(k-1)]$.
\end{lemma}

\begin{proof}
Let $(u(k))_{k=1}^\infty$ be given. Then, from \eqref{eq:V=sumfx*fxhf'xhx*xh}, \eqref{eq:xh=xh/u}, and \eqref{eq:f'xhf'xh=f'xhf'xh}, we have $V(\mathbf{x}(k))-V(\mathbf{x}(k-1))=-\sum_{i\in u(k)}f_i(\hat{x}_i(k))-f_i(\hat{x}_i(k-1))-f_i'(\hat{x}_i(k-1))(\hat{x}_i(k)-\hat{x}_i(k-1))+(f_i'(\hat{x}_i(k-1))-f_i'(\hat{x}_i(k)))\hat{x}_i(k)$ $\forall k\in\mathbb{P}$. Due to \eqref{eq:f'xhf'xh=f'xhf'xh} and \eqref{eq:xh<=xh<=xh<=xh}, $\sum_{i\in u(k)}(f_i'(\hat{x}_i(k-1))-f_i'(\hat{x}_i(k)))\hat{x}_i(k)=(f_{u_1(k)}'(\hat{x}_{u_1(k)}(k-1))-f_{u_1(k)}'(\hat{x}_{u_1(k)}(k)))(\hat{x}_{u_1(k)}(k)-\hat{x}_{u_2(k)}(k))\ge0$. This, along with \eqref{eq:fy>=fxf'xyx}, implies $V(\mathbf{x}(k))-V(\mathbf{x}(k-1))\le0$ $\forall k\in\mathbb{P}$. Now let $k\in\mathbb{P}$ and $\mathbf{x}(k-1)\in\mathcal{X}^N$ be given. By Lemma~\ref{lem:exisuniqz}, there exists a unique $x_{\text{eq}}\in\mathcal{X}$ such that $\sum_{i\in u(k)}f_i'(x_{\text{eq}})=\sum_{i\in u(k)}f_i'(\hat{x}_i(k))$. Also, $x_{\text{eq}}\in[\hat{x}_{u_1(k)}(k),\hat{x}_{u_2(k)}(k)]$. Let $\mathbf{x}_{\text{eq}}\in\mathcal{X}^N$ be such that its $i$th entry is $x_{\text{eq}}$ if $i\in u(k)$ and $\hat{x}_i(k-1)$ otherwise. Then, it follows from \eqref{eq:V=sumfx*fxhf'xhx*xh}, \eqref{eq:xh=xh/u}, and \eqref{eq:fy>=fxf'xyx} that $V(\mathbf{x}(k))-V(\mathbf{x}_{\text{eq}})=\sum_{i\in u(k)}f_i(x_{\text{eq}})-f_i(\hat{x}_i(k))-f_i'(\hat{x}_i(k))(x_{\text{eq}}-\hat{x}_i(k))\ge0$. Because $f_i(y)-f_i(x)-f_i'(x)(y-x)$ strictly increases with $|y-x|$ for each fixed $y\in\mathcal{X}$ $\forall i\in\mathcal{V}$ and because of \eqref{eq:f'xhf'xh=f'xhf'xh} and \eqref{eq:xh<=xh<=xh<=xh}, the second claim is true.
\end{proof}

Lemma~\ref{lem:PBVnonincr} says that the value of $V$ can never increase. In addition, the closer $\hat{x}_{u_1(k)}(k)$ and $\hat{x}_{u_2(k)}(k)$ get, the larger the value of $V$ drops, and the drop is maximized when $\hat{x}_{u_1(k)}(k)$ and $\hat{x}_{u_2(k)}(k)$ are equalized. These observations suggest that perhaps it is possible to design an algorithm that only forces $\hat{x}_{u_1(k)}(k)$ and $\hat{x}_{u_2(k)}(k)$ to approach each other (as opposed to becoming equal) to the detriment of a smaller drop in the value of $V$, but at the benefit of not having to share the $f_i$'s. The following algorithm, referred to as {\em Pairwise Bisectioning} (PB), shows that this is indeed the case and utilizes a bisection step that allows $\hat{x}_{u_1(k)}(k)$ and $\hat{x}_{u_2(k)}(k)$ to get arbitrarily close:

\begin{algorithm}[Pairwise Bisectioning]\label{alg:PB}
\begin{algorithminit}{}
\item Each node $i\in\mathcal{V}$ computes $x_i^*\in\mathcal{X}$, creates variables $\hat{x}_i,a_i,b_i\in\mathcal{X}$, and sets $\hat{x}_i\leftarrow x_i^*$.
\end{algorithminit}
\begin{algorithmoper}{At each iteration:}
\item A node with one or more one-hop neighbors, say, node $i$, initiates the iteration and selects a one-hop neighbor, say, node $j$, to gossip. Node $i$ transmits $\hat{x}_i$ to node $j$. Node $j$ sets $a_j\leftarrow\min\{\hat{x}_i,\hat{x}_j\}$ and $b_j\leftarrow\max\{\hat{x}_i,\hat{x}_j\}$ and transmits $\hat{x}_j$ to node $i$. Node $i$ sets $a_i\leftarrow\min\{\hat{x}_i,\hat{x}_j\}$ and $b_i\leftarrow\max\{\hat{x}_i,\hat{x}_j\}$. Nodes $i$ and $j$ select the number of bisection rounds $R\in\mathbb{P}$.
\item Repeat the following $R$ times: Node $j$ transmits $f_j'(\frac{a_j+b_j}{2})-f_j'(\hat{x}_j)$ to node $i$. Node $i$ tests if $f_j'(\frac{a_j+b_j}{2})-f_j'(\hat{x}_j)+f_i'(\frac{a_i+b_i}{2})-f_i'(\hat{x}_i)\ge0$. If so, node $i$ sets $b_i\leftarrow\frac{a_i+b_i}{2}$ and transmits LEFT to node $j$, and node $j$ sets $b_j\leftarrow\frac{a_j+b_j}{2}$. Otherwise, node $i$ sets $a_i\leftarrow\frac{a_i+b_i}{2}$ and transmits RIGHT to node $j$, and node $j$ sets $a_j\leftarrow\frac{a_j+b_j}{2}$. End repeat.
\item Node $j$ transmits $f_j'(c_j)-f_j'(\hat{x}_j)$ to node $i$, where $c_j=\bigl\{\begin{smallmatrix}a_j & \text{if $\hat{x}_j\le a_j$}\\ b_j & \text{if $\hat{x}_j\ge b_j$}\end{smallmatrix}$. Node $i$ tests if $\Bigl(f_j'(c_j)-f_j'(\hat{x}_j)+f_i'(c_i)-f_i'(\hat{x}_i)\Bigr)(\hat{x}_i-\frac{a_i+b_i}{2})\ge0$, where $c_i=\bigl\{\begin{smallmatrix}a_i & \text{if $\hat{x}_i\le a_i$}\\ b_i & \text{if $\hat{x}_i\ge b_i$}\end{smallmatrix}$. If so, node $i$ sets $\hat{x}_i\leftarrow(f_i')^{-1}(f_i'(\hat{x}_i)-f_j'(c_j)+f_j'(\hat{x}_j))$ and node $j$ sets $\hat{x}_j\leftarrow c_j$. Otherwise, node $i$ transmits $f_i'(c_i)-f_i'(\hat{x}_i)$ to node $j$ and sets $\hat{x}_i\leftarrow c_i$, and node $j$ sets $\hat{x}_j\leftarrow(f_j')^{-1}(f_j'(\hat{x}_j)-f_i'(c_i)+f_i'(\hat{x}_i))$.
\end{algorithmoper}
\end{algorithm}

Notice that Step~1 of PB is identical to that of PE except that each node $i\in\mathcal{V}$ creates two additional variables, $a_i$ and $b_i$, which are used in Step~2 to represent the initial bracket $[a_i,b_i]=[a_j,b_j]=[\min\{\hat{x}_i,\hat{x}_j\},\max\{\hat{x}_i,\hat{x}_j\}]$ for bisection purposes. Step~3 describes execution of the bisection method, where $R\in\mathbb{P}$ denotes the number of bisection rounds, which may be different for each iteration (e.g., a large $R$ may be advisable when $\hat{x}_i$ and $\hat{x}_j$ are very different). Observe that upon completing Step~3, $x_{\text{eq}}\in[a_i,b_i]=[a_j,b_j]\subset[\min\{\hat{x}_i,\hat{x}_j\},\max\{\hat{x}_i,\hat{x}_j\}]$ and $b_i-a_i=b_j-a_j=\frac{1}{2^R}|\hat{x}_j-\hat{x}_i|$, where $x_{\text{eq}}$ denotes the equalized value of $\hat{x}_i$ and $\hat{x}_j$ if PE were used. Moreover, upon completing Step~4, $x_{\text{eq}}\in[\min\{\hat{x}_i,\hat{x}_j\},\max\{\hat{x}_i,\hat{x}_j\}]\subset[a_i,b_i]=[a_j,b_j]$, where $\hat{x}_i$ and $\hat{x}_j$ here represent new values. Therefore, upon completing each iteration $k\in\mathbb{P}$,
\begin{align}
|\hat{x}_{u_1(k)}(k)-\hat{x}_{u_2(k)}(k)|\le\frac{1}{2^R}|\hat{x}_{u_1(k)}(k-1)-\hat{x}_{u_2(k)}(k-1)|,\quad\forall k\in\mathbb{P}.\label{eq:|xhxh|<=12R|xhxh|}
\end{align}
Finally, note that unlike PE which requires two real-number transmissions per iteration, PB requires as many as $3+R$ or $4+R$. However, it allows the nodes to never share their $f_i$'s.

The following theorem establishes the asymptotic convergence of PB under Assumption~\ref{asm:PEnetconn}:

\begin{theorem}\label{thm:PBasymconv}
Consider the use of PB described in Algorithm~\ref{alg:PB}. Suppose Assumptions~\ref{asm:fi} and~\ref{asm:PEnetconn} hold. Then, \eqref{eq:limV=0} and \eqref{eq:limxh=x*} hold.
\end{theorem}

\begin{proof}
See Appendix~\ref{ssec:proofthmPBasymconv}.
\end{proof}

As it follows from the above, PB represents an alternative to PE, which is useful when nodes are either unable, or unwilling, to share their $f_i$'s. Although not pursued here, it is straightforward to see that PE and PB may be combined, so that equalizing is used when one of the gossiping nodes can send the other its $f_i$, and approaching is used when none of them can.

\section{Conclusion}\label{sec:concl}

In this paper, based on the ideas of conservation and dissipation, we have developed PE and PB, two non-gradient-based gossip algorithms that enable nodes to cooperatively solve a class of convex optimization problems over networks. Using Lyapunov stability theory and the convexity structure, we have shown that PE and PB are asymptotically convergent, provided that the gossiping pattern is sufficiently rich. We have also discussed several salient features of PE and PB, including their comparison with the subgradient algorithms and their connection with distributed averaging.

\appendix
\section{Appendix}\label{sec:app}

\subsection{Proof of Theorem~\ref{thm:PEasymconv}}\label{ssec:proofthmPEasymconv}

Suppose Assumption~\ref{asm:fi} holds and let $(u(k))_{k=1}^\infty$ satisfying Assumption~\ref{asm:PEnetconn} be given. Consider the following lemmas:

\begin{lemma}\label{lem:exisgamm}
Suppose Assumption~\ref{asm:fi} holds. Then, $\forall[a,b]\subset\mathcal{X}$, there exists a continuous and strictly increasing function $\gamma:[0,\infty)\rightarrow[0,\infty)$ satisfying $\gamma(0)=0$ and $\lim_{d\rightarrow\infty}\gamma(d)=\infty$, such that $\forall\eta>0$, $\forall i\in\mathcal{V}$, $\forall(x,y)\in[a,b]^2$, $f_i(y)-f_i(x)-f_i'(x)(y-x)\le\eta$ implies $|y-x|\le\gamma^{-1}(\eta)$.
\end{lemma}

\begin{proof}
Let $[a,b]\subset\mathcal{X}$. For each $i\in\mathcal{V}$, define $g_i:[a,b]^2\rightarrow\mathbb{R}$ as $g_i(x,y)=f_i(y)-f_i(x)-f_i'(x)(y-x)$. Due to Assumption~\ref{asm:fi} and \eqref{eq:fy>=fxf'xyx}, $g_i(x,y)\ge0$ $\forall(x,y)\in[a,b]^2$, where the equality holds if and only if $x=y$. Moreover, since $f_i'$ is strictly increasing and $g_i(x,y)$ can be written as $g_i(x,y)=\int_x^y(f_i'(t)-f_i'(x))dt$, $g_i(x,y)$ is strictly increasing with $|y-x|$ for each fixed $x\in[a,b]$. Furthermore, because $f_i$ and $f_i'$ are continuous, $g_i$ is continuous. Next, for each $d\in[0,b-a]$, let $\mathcal{K}(d)=\{(x,y)\in[a,b]^2:|y-x|=d\}$. Also, for each $i\in\mathcal{V}$, define $\gamma_i:[0,b-a]\rightarrow\mathbb{R}$ as $\gamma_i(d)=\min_{(x,y)\in\mathcal{K}(d)}g_i(x,y)$. Due to the compactness of $\mathcal{K}(d)$ $\forall d\in[0,b-a]$ and the continuity of $g_i$, $\gamma_i$ is well-defined and continuous. In addition, since $g_i(x,y)=0$ $\forall(x,y)\in\mathcal{K}(0)$, $\gamma_i(0)=0$. Now pick any $d_1$ and $d_2$ such that $0\le d_1<d_2\le b-a$. Let $(x_2,y_2)\in\mathcal{K}(d_2)$ be such that $\gamma_i(d_2)=g_i(x_2,y_2)$. If $y_2>x_2$, then $y_2-x_2=d_2$. In this case, $\exists y_1\in[x_2,y_2)$ such that $y_1-x_2=d_1$. Since $g_i(x_2,y)$ is strictly increasing with $y$ for $y\ge x_2$, we have $\gamma_i(d_1)\le g_i(x_2,y_1)<g_i(x_2,y_2)=\gamma_i(d_2)$. Similarly, if $y_2<x_2$, we also have $\gamma_i(d_1)<\gamma_i(d_2)$. Hence, $\gamma_i$ is strictly increasing. Finally, define $\gamma:[0,\infty)\rightarrow[0,\infty)$ as $\gamma(d)=\bigl\{\begin{smallmatrix}\min_{i\in\mathcal{V}}\gamma_i(d) & \text{if $d\in[0,b-a]$}\\ \min_{i\in\mathcal{V}}\gamma_i(b-a)+d-(b-a) & \text{if $d\in(b-a,\infty)$}\end{smallmatrix}$. Note that $\gamma(0)=0$ since $\gamma_i(0)=0$ $\forall i\in\mathcal{V}$, and that $\lim_{d\rightarrow\infty}\gamma(d)=\infty$. Moreover, since $\gamma_i$ is continuous and strictly increasing $\forall i\in\mathcal{V}$, so is $\gamma$ on $[0,b-a]$. Also, observe that $\gamma$ is continuous and strictly increasing on $[b-a,\infty)$. Thus, $\gamma$ is continuous and strictly increasing. Now let $\eta>0$, $i\in\mathcal{V}$, and $(x,y)\in[a,b]^2$. Suppose $g_i(x,y)\le\eta$. If $\eta\le\gamma(b-a)$, then $|y-x|\le\gamma^{-1}(\eta)$ because $\gamma(|y-x|)\le\gamma_i(|y-x|)\le g_i(x,y)\le\eta$. If $\eta>\gamma(b-a)$, then $|y-x|\le b-a<\gamma^{-1}(\eta)$.
\end{proof}

\begin{lemma}\label{lem:exisbeta}
Suppose Assumption~\ref{asm:fi} holds. Then, $\forall[a,b]\subset\mathcal{X}$, $\exists\beta\in(0,\infty)$ such that $\forall i\in\mathcal{V}$, $\forall(x,y)\in[a,b]^2$, $f_i(y)-f_i(x)-f_i'(x)(y-x)\le\beta|y-x|$.
\end{lemma}

\begin{proof}
Let $[a,b]\subset\mathcal{X}$ and $\beta=1+2\max_{j\in\mathcal{V}}|f_j'(b)|$. Obviously, $\beta>0$, and by Assumption~\ref{asm:fi}, $\beta<\infty$. Let $i\in\mathcal{V}$ and $(x,y)\in[a,b]^2$. Since $f_i$ is continuously differentiable, by the Mean Value Theorem, $\exists c$ between $x$ and $y$ such that $f_i(y)-f_i(x)=f_i'(c)(y-x)$. This, along with the triangle inequality and the fact that $f_i'$ is strictly increasing, implies that $f_i(y)-f_i(x)-f_i'(x)(y-x)=(f_i'(c)-f_i'(x))(y-x)\le|f_i'(c)-f_i'(x)|\cdot|y-x|\le(|f_i'(c)|+|f_i'(x)|)|y-x|\le2|f_i'(b)|\cdot|y-x|\le\beta|y-x|$.
\end{proof}

Let $a=\min_{i\in\mathcal{V}}\hat{x}_i(0)$ and $b=\max_{i\in\mathcal{V}}\hat{x}_i(0)$. Then, it follows from Proposition~\ref{pro:xhwelldef} that $\hat{x}_i(k)\in[a,b]\subset\mathcal{X}$ $\forall k\in\mathbb{N}$ $\forall i\in\mathcal{V}$ and from \eqref{eq:sumf'xh=0} and Lemma~\ref{lem:exisuniqz} that $x^*\in[a,b]$. By Lemma~\ref{lem:exisgamm}, there exists a continuous and strictly increasing function $\gamma:[0,\infty)\rightarrow[0,\infty)$ satisfying $\gamma(0)=0$ and $\lim_{d\rightarrow\infty}\gamma(d)=\infty$, such that $\forall\eta>0$, $\forall i\in\mathcal{V}$, $\forall(x,y)\in[a,b]^2$, $f_i(y)-f_i(x)-f_i'(x)(y-x)\le\eta$ implies $|y-x|\le\gamma^{-1}(\eta)$. Also, by Lemma~\ref{lem:exisbeta}, $\exists\beta\in(0,\infty)$ such that $\forall i\in\mathcal{V}$, $\forall(x,y)\in[a,b]^2$, $f_i(y)-f_i(x)-f_i'(x)(y-x)\le\beta|y-x|$. From Lemma~\ref{lem:PEVnonincr}, $(V(\mathbf{x}(k)))_{k=0}^\infty$ is nonnegative and non-increasing. Thus, $\exists c\ge0$ such that $\lim_{k\rightarrow\infty}V(\mathbf{x}(k))=c$. To show that $c$ must be zero, assume, to the contrary, that $c>0$. Let $\epsilon>0$ be given by $\epsilon=\gamma(\frac{c}{4\beta N^2})$. Then, $\exists k_1\in\mathbb{N}$ such that
\begin{align}
c\le V(\mathbf{x}(k))<c+\epsilon,\quad\forall k\ge k_1.\label{eq:c<=V<ce}
\end{align}
Due to \eqref{eq:c<=V<ce}, $V(\mathbf{x}(k-1))-V(\mathbf{x}(k))<\epsilon$ $\forall k\ge k_1+1$. Hence, from \eqref{eq:fy>=fxf'xyx} and \eqref{eq:VV=sumfxhfxhf'xhxhxh}, $f_i(\hat{x}_i(k))-f_i(\hat{x}_i(k-1))-f_i'(\hat{x}_i(k-1))(\hat{x}_i(k)-\hat{x}_i(k-1))<\epsilon$ $\forall k\ge k_1+1$ $\forall i\in u(k)$. As a result, $|\hat{x}_i(k)-\hat{x}_i(k-1)|\le\gamma^{-1}(\epsilon)$ $\forall k\ge k_1+1$ $\forall i\in u(k)$. Because of this and \eqref{eq:xh=xh},
\begin{align}
|\hat{x}_i(k)-\hat{x}_j(k)|\le2\gamma^{-1}(\epsilon),\quad\forall k\ge k_1,\;\forall i,j\in u(k+1).\label{eq:|xhxh|<=2ginve}
\end{align}
Now suppose $\max_{i\in\mathcal{V}}\hat{x}_i(k_1)-\min_{i\in\mathcal{V}}\hat{x}_i(k_1)>2(N-1)\gamma^{-1}(\epsilon)$. Then, $\exists p,q\in\mathcal{V}$ such that $\hat{x}_q(k_1)-\hat{x}_p(k_1)>2\gamma^{-1}(\epsilon)$ and $\mathcal{C}_1\cup\mathcal{C}_2=\mathcal{V}$, where $\mathcal{C}_1=\{i\in\mathcal{V}:\hat{x}_i(k_1)\le\hat{x}_p(k_1)\}$ and $\mathcal{C}_2=\{i\in\mathcal{V}:\hat{x}_i(k_1)\ge\hat{x}_q(k_1)\}$. Next, we show by induction that $\forall k\ge k_1$, $\hat{x}_i(k)\le\hat{x}_p(k_1)$ $\forall i\in\mathcal{C}_1$ and $\hat{x}_i(k)\ge\hat{x}_q(k_1)$ $\forall i\in\mathcal{C}_2$. Clearly, the statement is true for $k=k_1$. For $k\ge k_1+1$, suppose $\hat{x}_i(k-1)\le\hat{x}_p(k_1)$ $\forall i\in\mathcal{C}_1$ and $\hat{x}_i(k-1)\ge\hat{x}_q(k_1)$ $\forall i\in\mathcal{C}_2$. Then, due to \eqref{eq:|xhxh|<=2ginve}, $\forall i\in\mathcal{C}_1$, $\forall j\in\mathcal{C}_2$, $\{i,j\}\neq u(k)$, i.e., $u(k)\subset\mathcal{C}_1$ or $u(k)\subset\mathcal{C}_2$. It follows from \eqref{eq:xh=sumf'invsumf'xhifiinuxh} and Lemma~\ref{lem:exisuniqz} that $\hat{x}_i(k)\le\hat{x}_p(k_1)$ $\forall i\in\mathcal{C}_1$ and $\hat{x}_i(k)\ge\hat{x}_q(k_1)$ $\forall i\in\mathcal{C}_2$, completing the induction. Due again to \eqref{eq:|xhxh|<=2ginve}, we have $\forall i\in\mathcal{C}_1$, $\forall j\in\mathcal{C}_2$, $\{i,j\}\neq u(k)$ $\forall k\ge k_1+1$, which violates Assumption~\ref{asm:PEnetconn}. Consequently, $\max_{i\in\mathcal{V}}\hat{x}_i(k_1)-\min_{i\in\mathcal{V}}\hat{x}_i(k_1)\le2(N-1)\gamma^{-1}(\epsilon)$. It follows from \eqref{eq:sumf'xh=0} and Lemma~\ref{lem:exisuniqz} that $|x^*-\hat{x}_i(k_1)|\le\max_{j\in\mathcal{V}}\hat{x}_j(k_1)-\min_{j\in\mathcal{V}}\hat{x}_j(k_1)\le2(N-1)\gamma^{-1}(\epsilon)$ $\forall i\in\mathcal{V}$. Hence, $V(\mathbf{x}(k_1))\le\beta\sum_{i\in\mathcal{V}}|x^*-\hat{x}_i(k_1)|\le\beta\cdot N\cdot2(N-1)\gamma^{-1}(\epsilon)<c$, which contradicts \eqref{eq:c<=V<ce}. Therefore, $c=0$, i.e., \eqref{eq:limV=0} holds, implying that \eqref{eq:limxh=x*} is satisfied.

\subsection{Proof of Theorem~\ref{thm:PBasymconv}}\label{ssec:proofthmPBasymconv}

The proof is similar to that of Theorem~\ref{thm:PEasymconv}. Let $a$, $b$, $\gamma$, and $\beta$ be as defined in Appendix~\ref{ssec:proofthmPEasymconv}. Then, due to \eqref{eq:xh=xh/u}, \eqref{eq:xh<=xh<=xh<=xh}, \eqref{eq:sumf'xh=0}, and Lemma~\ref{lem:exisuniqz}, we have $\hat{x}_i(k)\in[a,b]$ $\forall k\in\mathbb{N}$ $\forall i\in\mathcal{V}$ and $x^*\in[a,b]$. From Lemma~\ref{lem:PBVnonincr}, $\lim_{k\rightarrow\infty}V(\mathbf{x}(k))=c$ for some $c\ge0$. To show that $c=0$, assume to the contrary that $c>0$ and let $\epsilon$ be as defined in~\ref{ssec:proofthmPEasymconv}. Then, \eqref{eq:c<=V<ce} holds for some $k_1\in\mathbb{N}$. It follows from the proof of Lemma~\ref{lem:PBVnonincr} that $f_i(\hat{x}_i(k))-f_i(\hat{x}_i(k-1))-f_i'(\hat{x}_i(k-1))(\hat{x}_i(k)-\hat{x}_i(k-1))\le V(\mathbf{x}(k-1))-V(\mathbf{x}(k))<\epsilon$ $\forall k\ge k_1+1$ $\forall i\in u(k)$. Thus, $|\hat{x}_i(k)-\hat{x}_i(k-1)|\le\gamma^{-1}(\epsilon)$ $\forall k\ge k_1+1$ $\forall i\in u(k)$. This, along with \eqref{eq:|xhxh|<=12R|xhxh|} and the fact that $R\in\mathbb{P}$, implies $|\hat{x}_i(k)-\hat{x}_j(k)|\le\frac{2\gamma^{-1}(\epsilon)}{1-\frac{1}{2^R}}\le4\gamma^{-1}(\epsilon)$ $\forall k\ge k_1$ $\forall i,j\in u(k+1)$. Then, using the same idea as in~\ref{ssec:proofthmPEasymconv}, it can be shown that $\max_{i\in\mathcal{V}}\hat{x}_i(k_1)-\min_{i\in\mathcal{V}}\hat{x}_i(k_1)\le4(N-1)\gamma^{-1}(\epsilon)$. This leads to $V(\mathbf{x}(k_1))<c$, which contradicts \eqref{eq:c<=V<ce}. Therefore, \eqref{eq:limV=0} and \eqref{eq:limxh=x*} hold.

\bibliographystyle{IEEEtran}
\bibliography{paper}

\end{document}